\documentclass[11pt]{article}
\usepackage[margin=1in]{geometry}
\usepackage{natbib}
\usepackage[symbol]{footmisc} 
\usepackage[hidelinks]{hyperref}
\usepackage{booktabs}
\usepackage{tabularx}
\usepackage{graphicx}%
\usepackage{eso-pic}
\usepackage{enumitem}
\usepackage{amsfonts,amsmath,amssymb,amsthm,mathtools}%
\usepackage[title]{appendix}%
\usepackage{lmodern}
\usepackage{caption} % Anonymized submission
\usepackage{authblk}

\usepackage{times}

\usepackage{tikz}
\usepackage{amsmath}
\usepackage{amssymb}
\usepackage{mathtools}
\usepackage{amsmath}
\usepackage{float}
\usepackage{algorithm}
\usepackage{algpseudocode}
\newtheorem{fact}{Fact}

\newcommand{\real}{\mathbb{R}}

\newcommand{\cX}{\mathcal{X}}

\newcommand*{\infn}[1]{\left\|{#1}\right\|_{\infty}}

\newcommand{\expec}{\mathbb{E}}

\newtheorem{assumption}{Assumption}
\newtheorem{theorem}{Theorem}
\newtheorem{lemma}{Lemma}

\date{}

\begin{document}
\title{\LARGE\bfseries
Optimal Oracle Complexity of Dual Halpern Iteration in Normed Spaces  
}

% \title{\LARGE\bfseries
% Dual Halpern Iteration in Normed Spaces:\\
% Oracle Complexities and Robust Error Bounds
% }

% \title{\LARGE\bfseries
% Dual Halpern Iteration is Optimal for Stochastic Fixed-Point Problems
% }
% \title{\bfseries Dual Halpern Iteration in General Normed Space:
% robust Error Bounds Under Uncertainty}
\author{}
\date{}% \author{Jongmin Lee}
\maketitle

\vspace{-3em}

\noindent
\begin{minipage}[t]{0.65\textwidth}
    \textbf{Jongmin Lee}\\
    \textit{ University of Pennsylvania}
\end{minipage}%
\hfill
\begin{minipage}[t]{0.33\textwidth}
\raggedleft
    \small  jongmlee@engineering.upenn.edu
\end{minipage}

\begin{abstract}
In this work, we study the convergence and error bounds of dual Halpern
iteration for nonexpansive operators in normed spaces. First, we show
that dual Halpern iteration exhibits the same optimal residual rate
as Halpern iteration in the deterministic setting through a direct
convergence analysis. With inexact operator evaluations, we also conduct
a convergence analysis in normed spaces and derive the robust error bound
of inexact dual Halpern iteration. Based on this robustness result, we study the oracle
complexity of stochastic dual Halpern iteration in
$(\real^d,\|\cdot\|_p)$ for $p\in(1,\infty]$, using an unbiased stochastic oracle
with bounded variance $\sigma^2$. Notably, for $p\in[2,\infty]$,
our method achieves an oracle complexity of
$O\left(\epsilon^{-1}+\sigma^2\epsilon^{-3}\right)$ in its dependence on $\epsilon$ and $\sigma$,
which matches the lower bound for the case $p=\infty$ without a logarithmic
gap in the dependence on $\epsilon$.
\end{abstract}

% \AddToShipoutPictureFG*{%
%   \AtPageLowerLeft{%
%     \put(0,40){\makebox[\paperwidth][c]{Institution 1}}%
%   }%
% }

% \vspace*{1cm}
% % \maketitle

% \title{\Large\bfseries
% Dual Halpern Iteration in General Normed Space: Robust Error Bounds Under Uncertainty
% }

% \maketitle

% \vspace{-1.5em}

% \noindent
% \begin{minipage}[t]{0.75\textwidth}
%     \textbf{Jongmin Lee}\\
%     \textit{University of Pennsylvania}
% \end{minipage}%
% \hfill
% \begin{minipage}[t]{0.23\textwidth}
%     \raggedleft
%     \small jongmlee@engineering.upenn.edu
% % \end{minipage}

% \maketitle
% \vspace{-3em}

% \noindent
% \begin{minipage}[t]{0.7\textwidth}
% \textbf{Jongmin Lee}\\
% \textit{University of Pennsylvania}
% \end{minipage}%
% \hfill
% \begin{minipage}[t]{0.28\textwidth}
% \raggedleft
% jongmlee@engineering.upenn.edu
% \end{minipage}

% \begin{center}
%     {\LARGE
%      Dual Halpern Iteration in General Normed Space:
%     Robust Error Bounds Under Uncertainty
%     \par}

%     \vspace{1.2em}

%     {\large Jongmin Lee \qquad\qquad\qquad\qquad\qquad\texttt{jongmlee@engineering.upenn.edu} \par}

%     \vspace{0.4em}
    
%     {\large\itshape
%     \!\!\!\!\!\!\!\!\!\!University of Pennsylvania
%     \par}
% \end{center}

% \vspace{2em}
% \begin{abstract}%
%   This is a great paper and it has a concise abstract.%
% \end{abstract}

\section{Introduction}
To solve the fixed-point problem associated with a nonexpansive operator $T:\cX\to\cX$,
\[
\underset{x\in\cX}{\text{find }} \,\,\, T(x)=x,
\]
% provides a general rubric of formulating problem in optimization, reinforcement
% learning, and machine learning.
Halpern iteration reduces the fixed-point residual
$\|Tx-x\|$ with optimal oracle complexity in the deterministic setting.
In the stochastic oracle setup, Halpern-based
methods further achieved near-optimal complexity for Euclidean spaces and certain non-Euclidean spaces
using recursive regularization techniques, but there still exists
a logarithmic gap to optimal stochastic oracle complexity.

On the other hand, \emph{dual Halpern iteration}
was recently introduced with its optimality in the deterministic setting established through
a new type of duality theorem, and it also achieved near-optimal stochastic oracle complexity
in Euclidean spaces without variance reduction techniques.
Despite these promising convergence results, two limitations remain
for ensuring its performance: (i) prior stochastic fixed-point results require an additional
mean-square nonexpansiveness assumption on the stochastic oracle, and (ii) 
prior analyses do not extend to general normed spaces whose
norms are not induced by an inner product.

\paragraph{Contributions.}
In this work, we establish optimal oracle complexity bounds for dual Halpern
iteration for nonexpansive operators in normed spaces. First, in the
deterministic setup, we obtain the optimal residual bound $O(1/K)$ through a direct
convergence analysis, where $K$ is the number of iterations
(Theorem~\ref{thm:dh}). With inexact operator evaluations,
i.e., $\widetilde T x=Tx+e$, we show that the error term in the bound
for dual Halpern iteration grows slowly as
$O(\log(K+1)\|e\|)$ (Theorem~\ref{thm:idh}), compared with the error bound
$O(K\|e\|)$ of inexact Halpern iteration.
Building on this robustness result, we analyze stochastic dual Halpern
iteration in $(\real^d,\|\cdot\|_p)$ for $p\in(1,\infty]$, using an
unbiased stochastic oracle with uniformly bounded variance $\sigma^2$.
In particular, for $p\in[2,\infty]$, stochastic dual Halpern iteration
yields an oracle complexity of
\[
O\left(
\frac{\kappa}{\epsilon}
+
\frac{\kappa(p-1)\sigma^2}{\epsilon^3}
\right)
\quad \text{for } p\in[2,\infty),
\qquad
O\left(
\frac{\kappa}{\epsilon}
+
\frac{\kappa\sigma^2\log(2d)}{\epsilon^3}
\right)
\quad \text{for } p=\infty,
\]
where $\kappa=\|x^0-x^\star\|_p$ for a fixed point $x^\star$
(Theorem~\ref{thm:sdh}).
Notably, for $p=\infty$, our oracle complexity is optimal up to the
logarithmic dimension factor, and this improves upon the prior state-of-the-art oracle
complexity results by logarithmic factors in $1/\epsilon$ \cite{diakonikolas2026oracle}.

\subsection{Prior works}
\paragraph{Oracle complexity of Halpern and dual Halpern iteration}

Halpern iteration \citep{halpern1967fixed} has been extensively studied in the literature on minimax optimization and fixed-point problems, particularly from the perspective of oracle complexity.
Its $O(1/\epsilon)$ oracle complexity was established by
\cite{sabach2017first} for appropriate anchoring coefficients, with
exact optimality in Hilbert spaces established by
\cite{Lieder2021halpern, park2022exact}. Improved bounds in normed
spaces were studied by \cite{contreras2022optimal, bravo2026minimax}.
In stochastic settings, variance-reduced Halpern methods have been
studied in Euclidean spaces by \cite{cai2022stochastic, cai2024variance},
including the finite-sum setting. Using recursive regularization,
\cite{chen2024near} achieves near-optimal
$\widetilde{O}(1/\epsilon^{2})$ oracle complexity for stochastic
monotone Lipschitz root-finding in Euclidean spaces.
For finite-dimensional normed spaces,
\cite{bravo2026stochastic, diakonikolas2026solving} establish
$\widetilde{O}(1/\epsilon^{5})$ oracle complexity using stochastic
Halpern iteration with minibatching and with recursive variance
reduction, respectively. Recently, \cite{diakonikolas2026oracle}
achieved $\widetilde{O}(1/\epsilon^{3})$ oracle complexity for
nonexpansive self-maps  for $p\in[2,\infty]$
through recursive anchoring combined with median-of-means estimation,
with near-optimality established for $\ell_\infty^d$.

Dual Halpern iteration was first introduced in Hilbert spaces through
the H-duality theorem \cite{yoon2024optimal}. The H-duality theorem
guarantees the same residual bound for H-dual fixed-point algorithms
when the proofs have the prescribed Lyapunov structure. So dual Halpern iteration
is optimal, as Halpern iteration is, and this H-duality
framework was further studied in
\cite{yoon2026toward, shu2026unified}.
Similarly, in general normed spaces, the optimal
$O(1/\epsilon)$ oracle complexity of dual Halpern iteration
was first established by the M-duality theorem \cite{lee2026m}. Compared to H-duality, M-duality
establishes equal normalized worst-case residual bounds for
corresponding M-dual fixed-point algorithms under appropriate
coefficient assumptions.
For stochastic fixed-point problems in finite-dimensional Euclidean
spaces, \cite{yoon2026direct} established
$O(1/\epsilon^3)$ oracle complexity using stochastic dual Halpern
iteration without recursive variance reduction or recursive
regularization.
\cite{yoon2026dual} improved this to a near-optimal
$\widetilde{O}(1/\epsilon^2)$ oracle complexity by optimizing the
minibatch schedule. However, these stochastic fixed-point analyses
of dual Halpern iteration require an additional mean-square
nonexpansiveness assumption on the stochastic oracle, which is not
required in the other near-optimal results
\cite{chen2024near, diakonikolas2026oracle}.

Regarding oracle complexity lower bounds, in the deterministic
setting, \cite{contreras2022optimal} establishes an
$\Omega(1/\epsilon)$ lower bound in $\ell_\infty(\mathbb{N})$
for general Mann iterations, while \cite{lee2024accelerating}
establishes a lower bound for methods satisfying
a span condition in finite-dimensional spaces.
In Euclidean spaces, \cite{diakonikolas2022potential} established
$\Omega(1/\epsilon)$ lower bounds within their 1-SCLI framework, while
\cite{park2022exact} established sharp lower bounds for
deterministic fixed-point algorithms.
In the stochastic setting, \cite{chen2024near} established an
$\Omega(1/\epsilon^2)$ lower bound for
smooth convex-concave minimax problems in Euclidean spaces.
For stochastic fixed-point problems,
\cite{bravo2026stochastic} established an
$\Omega(1/\epsilon^3)$ lower bound in $\ell_1^d$ under a span condition, with oracle variance bounded in the Euclidean norm, and
\cite{diakonikolas2026oracle} established an
$\Omega(1/\epsilon^3)$ lower bound in $\ell_\infty^d$ with sufficiently high dimension for
arbitrary randomized algorithms with oracle variance bounded in the supremum norm.
\subsection{Notation}
Write $(\cX,\|\cdot\|)$ for a normed space. An operator $T: \cX \rightarrow \cX$ is nonexpansive if $\|Tx-Ty\| \le \|x-y\|$ for all $x,y \in \cX$.
We denote the $\ell_p$ norms in $\mathbb{R}^d$ by $\|\cdot\|_p$ for $p\in[1,\infty]$, i.e.,
for $x = (x_1,\ldots,x_d) \in \mathbb{R}^d$:
$\|x\|_p = \left(\sum_{i=1}^d |x_i|^p\right)^{1/p}$
for $p \in [1,\infty)$, and
$\|x\|_\infty = \max_{1\le i\le d}|x_i|$.
We say $x^\star$ is a fixed point of the operator $T$ if $Tx^\star=x^\star$, and denote
$\|x^0-x^\star\|=\kappa$ for a given $(\cX,\|\cdot\|)$.

\section{Dual Halpern iteration}
\begin{assumption}\label{assum_fixp}
$T:\cX\to\cX$ is a nonexpansive operator on a normed space
$(\cX,\|\cdot\|)$ with a fixed point $x^\star$.
\end{assumption}

The \emph{Dual Halpern iteration} is
\begin{align*}
    x^{k+1}
    &=m_{K-k}Tx^k
    +\sum_{i=0}^{k-1}(m_{K-i}-m_{K-i-1})Tx^i
    +(1-m_K)x^0,
    \qquad 0\le k \le K-1,
\end{align*}
and equivalently,
\[
x^{k+1}=x^k+m_{K-k}(Tx^k-Tx^{k-1}),
\quad 0\le k \le K-1,
\]
where $Tx^{-1}=x^0$ and $K$ is the total number of iterations.

Dual Halpern iteration was first introduced in Hilbert spaces by a duality theorem \cite{yoon2024optimal}, and its optimality was also established in both Hilbert and normed spaces \cite{yoon2024optimal, lee2026m}. Roughly speaking, the duality theorems state that a fixed-point iteration and its corresponding ``dual'' iteration exhibit the same residual bound under appropriate assumptions. Thus the convergence rate and optimality of dual Halpern iteration follow directly from those of Halpern iteration \cite{Lieder2021halpern, contreras2022optimal}.

Here, we directly analyze the convergence of dual Halpern iteration under Assumption~\ref{assum_fixp} without relying on a duality theorem.

\begin{theorem}\label{thm:dh}
Assume Assumption~\ref{assum_fixp}. Let $K\ge1$ and $0=m_0\le m_k\le m_{k+1}\le1$
for $0\le k<K$. The Dual Halpern iteration exhibits the rate
\[
\|Tx^K-x^K\|
\le
2\left(
1-m_K+\sum_{i=1}^K(m_i-m_{i-1})\prod_{j=i}^K m_j
\right)\kappa,
\]
where $\kappa= \|x^0-x^\star\|$. Specifically, if $m_k=k/(k+2)$, the Dual Halpern iteration
exhibits the rate
\[
\|Tx^K-x^K\|\le\frac{8\kappa}{K+4}.
\]
\end{theorem}
\begin{proof}
First, we prove $\|Tx^{k+1}-x^\star\| \le \|x^0-x^\star\|$ by induction.  

$\|Tx^0-x^\star\|\le \|x^0-x^\star\|$ by the nonexpansiveness of the operator $T$. Then, we have
\begin{align*}
    \|Tx^{k+1}-x^\star\|
&\le \|x^{k+1}-x^\star\|\\
&=\left\|
m_{K-k}Tx^k
+\sum_{i=0}^{k-1}(m_{K-i}-m_{K-i-1})Tx^i
+(1-m_K)x^0-x^\star
\right\|\\
&=
\left\|
m_{K-k}(Tx^k-x^\star)
+\sum_{i=0}^{k-1}(m_{K-i}-m_{K-i-1})(Tx^i-x^\star)
+(1-m_K)(x^0-x^\star)
\right\|\\
&\le m_{K-k}\left\|Tx^k-x^\star\right\|
+\sum_{i=0}^{k-1}(m_{K-i}-m_{K-i-1})\left\|Tx^i-x^\star\right\|
+(1-m_K)\left\|x^0-x^\star\right\|\\
&\le \left\|x^0-x^\star\right\|,
\end{align*}
where the first inequality follows from nonexpansiveness, the second is from the properties of the norm, and the last inequality comes from induction.

Then, we also have 
\[\|Tx^{k+1}-x^0\|\le \|Tx^{k+1}-x^\star\|+\|x^\star-x^0\|
\le \|x^{k+1}-x^\star\|+\|x^\star-x^0\|
\le 2\|x^\star-x^0\|\].

If $\kappa=0$, the claim is immediate. Otherwise, normalize $\kappa=1$ by scaling the norm without loss of generality.
Now, we prove the following key lemma.
\begin{lemma}\label{lem_1}
For $0\le j+1< k+1\le K$,
\[
\|x^{k+1}-x^{j+1}\|
\le2\prod_{i=K-j-1}^K m_i.
\]
\end{lemma}

\begin{proof}
We prove the claim by induction on $j$. First, for $j=-1$,
\begin{align*}
\|x^{k+1}-x^0\|
&=
\left\|
m_{K-k}Tx^k
+\sum_{i=0}^{k-1}(m_{K-i}-m_{K-i-1})Tx^i
+(1-m_K)x^0-x^0
\right\|\\
&=
\left\|
m_{K-k}(Tx^k-x^0)
+\sum_{i=0}^{k-1}(m_{K-i}-m_{K-i-1})(Tx^i-x^0)
\right\|\\
&\le
2m_{K-k}
+2\sum_{i=0}^{k-1}(m_{K-i}-m_{K-i-1})\\
&=2m_K.
\end{align*}
Next, we have
\begin{align*}
\|x^{k+1}-x^{j+1}\|
&=
\Biggl\|
m_{K-k}Tx^k
+\sum_{i=0}^{k-1}(m_{K-i}-m_{K-i-1})Tx^i
+(1-m_K)x^0\\
&\qquad
-\left(
m_{K-j}Tx^j
+\sum_{i=0}^{j-1}(m_{K-i}-m_{K-i-1})Tx^i
+(1-m_K)x^0
\right)
\Biggr\|\\
&=
\left\|
m_{K-k}Tx^k
+\sum_{i=j+1}^{k-1}(m_{K-i}-m_{K-i-1})Tx^i
-m_{K-j-1}Tx^j
\right\|\\
&=
\left\|
m_{K-k}(Tx^k-Tx^j)
+\sum_{i=j+1}^{k-1}
(m_{K-i}-m_{K-i-1})(Tx^i-Tx^j)
\right\|\\
&\le
m_{K-k}\|x^k-x^j\|
+\sum_{i=j+1}^{k-1}
(m_{K-i}-m_{K-i-1})\|x^i-x^j\|\\
&\le
2m_{K-k}\prod_{\ell=K-j}^K m_\ell
+\sum_{i=j+1}^{k-1}
(m_{K-i}-m_{K-i-1})2\prod_{\ell=K-j}^K m_\ell\\
&=
2\prod_{\ell=K-j-1}^K m_\ell,
\end{align*}
where the first inequality follows from the triangle inequality
and nonexpansiveness and the second is from the induction hypothesis.
\end{proof}

Therefore,
\begin{align*}
&\|Tx^K-x^K\|\\
&=
\Biggl\|
m_1(Tx^{K-1}-Tx^K)
+\sum_{i=0}^{K-2}
(m_{K-i}-m_{K-i-1})(Tx^i-Tx^K)
+(1-m_K)(x^0-Tx^K)
\Biggr\|\\
&\le
m_1\|Tx^{K-1}-Tx^K\|
+\sum_{i=0}^{K-2}
(m_{K-i}-m_{K-i-1})\|Tx^i-Tx^K\|
+(1-m_K)\|x^0-Tx^K\|\\
&\le
m_1\|x^{K-1}-x^K\|
+\sum_{i=0}^{K-2}
(m_{K-i}-m_{K-i-1})\|x^i-x^K\|
+2(1-m_K)\\
&\le
2(1-m_K)
+2m_1\prod_{j=1}^K m_j
+2\sum_{i=0}^{K-2}
(m_{K-i}-m_{K-i-1})\prod_{j=K-i}^K m_j\\
&=
2\left(1-m_K+\sum_{i=1}^K(m_i-m_{i-1})\prod_{j=i}^K m_j\right),
\end{align*}
where the first inequality follows from the triangle inequality,
the second is from nonexpansiveness and the normalization,
and the third inequality comes from the previous lemma.

Lastly, let $m_r=r/(r+2)$ and define
\[
B_r
=
1-m_r+\sum_{i=1}^r(m_i-m_{i-1})\prod_{j=i}^r m_j.
\]
Then $B_0=1$ and, for $r\ge1$,
\begin{align*}
B_r
&=
1-m_r
+m_r\bigl(m_r-m_{r-1}+B_{r-1}-(1-m_{r-1})\bigr)\\
&=(1-m_r)^2+m_rB_{r-1}.
\end{align*}
We prove $B_r\le4/(r+4)$ by induction. If
$B_{r-1}\le4/(r+3)$, then
\[
B_r
\le
\frac{4}{(r+2)^2}+\frac{4r}{(r+2)(r+3)}
\le\frac{4}{r+4}.
\]
% where the last inequality follows from
% \[
% \frac{4}{r+4}
% -\frac{4r}{(r+2)(r+3)}
% -\frac{4}{(r+2)^2}
% =
% \frac{4r}{(r+2)^2(r+3)(r+4)}
% \ge0.
% \]
Thus $\|Tx^K-x^K\|\le2B_K\le8/(K+4)$.
Restoring the factor $\kappa$ proves the claimed bounds.
\end{proof}
% We note that while \cite{lee2026m} showed that converenge rate of Dual Halpern iteration under bounded assumption, and this result also transformed into our setup, by dimaeter of domain as $\|x^0-x^\star\|$ and proving $x^n- x^\star \le x^0- x^\star $, the did not directly analyze the Dual Halperniteration. Here we prove its rate by direct analysis, and this will be bases of analysis of inexact dual halpern iteration. Also, we notify that unlike Hilbert spaces, Halpern iteration and Dual Halpern iteraiton is not exact optimal by exactly mathcing lower bound. Notably, \cite{lee2026m} introduced algorithm which showed faster worst case residual bound than any Halpern or dual Halpern type mtehod.

\subsection{Comparison with the residual bound of Halpern iteration}

The \emph{Halpern iteration} \cite{halpern1967fixed} is
\[
x^{k+1}
=
(1-m_{k+1})x^0+m_{k+1}Tx^k,
\qquad k=0,\ldots,K-1.
\]
% and the following convergence result is known.
\begin{fact}[\cite{sabach2017first, contreras2022optimal}]\label{fact_1}
Assume Assumption~\ref{assum_fixp}. Let $K\ge1$ and $0=m_0\le m_k\le m_{k+1}\le1$
for $0\le k<K$. The Halpern iteration exhibits the rate
\[
\|Tx^K-x^K\|
\le
2\left(
1-m_K+\sum_{i=1}^K(m_i-m_{i-1})\prod_{j=i}^K m_j
\right)\kappa.
\]
Specifically, if $m_k=k/(k+2)$, the Halpern iteration
exhibits the rate
\[
\|Tx^K-x^K\|\le\frac{8\kappa}{K+4}.
\]
\end{fact}

As Theorem~\ref{thm:dh} and Fact~\ref{fact_1} show, with $m_k=k/(k+2)$,
Halpern iteration and dual Halpern iteration exhibit the same optimal residual bound under Assumption~\ref{assum_fixp}, matching the lower bound
up to constants \cite{contreras2022optimal, lee2024accelerating}.
But we note that the convergence analyses of the two iterations are seemingly different.
In particular, in Lemma~\ref{lem_1}, the structure of the Dual Halpern
iteration gives an upper bound on $\|x^k-x^i\|$ that is independent of $k$ for fixed $i$ and $K$,
in contrast to the analysis of Halpern iteration in \cite{contreras2022optimal}.

Lastly, we note that  Halpern iteration and  dual Halpern iteration are not exactly optimal in general normed spaces, while they are exactly optimal in Hilbert spaces by exactly matching the lower bound with appropriate anchoring coefficients \cite{Lieder2021halpern, park2022exact}. Specifically, \cite{lee2026m} introduced a new fixed-point iteration which showed a smaller worst-case residual bound than Halpern or dual Halpern type methods.

\section{Inexact Dual Halpern iteration}
When the operator values $Tx$ can only be approximated by
$\widetilde{T}x=Tx+e$, one is naturally led to consider
the inexact iteration.

Define the \emph{Inexact Dual Halpern iteration} as
\begin{align*}
x^{k+1}
&=x^k+m_{K-k}
  (\widetilde{T}x^k-\widetilde{T}x^{k-1}),
\end{align*}
for $k=0,\ldots,K-1$, where
$\widetilde{T}x^k=Tx^k+e_k$,
$\widetilde{T}x^{-1}=x^0$, and $K$ is the total number of iterations.

We now introduce our main result, which demonstrates the robust error bound of dual Halpern iteration.
\begin{theorem}\label{thm:idh}
Assume Assumption~\ref{assum_fixp}.
Let $0=m_0\le m_1\le\cdots\le m_K\le1$.
Then the Inexact Dual Halpern iteration exhibits the error bound
\[
\|Tx^K-x^K\|
\le
2(1-m_K+m_KW_K)\kappa
+2\sum_{i=0}^{K-1}W_{K-i}\|e_i\|,
\]
where
\[
W_0=0,
\qquad
W_i=m_i-m_{i-1}+m_{i-1}W_{i-1},
\quad i=1,\ldots,K.
\]
Specifically, if $m_k=k/(k+2)$, then
\[
\|Tx^K-x^K\|
\le
\frac{8\kappa}{K+4}
+4\sum_{i=0}^{K-1}\frac{\|e_i\|}{K-i+2}.
\]
\end{theorem}

\begin{proof}
Let $\Delta_i=m_i-m_{i-1}$ for $i=1,\ldots,K$.
Then we have
\begin{align*}
x^{k+1}
&=
m_{K-k}(Tx^k+e_k)
+\sum_{i=0}^{k-1}\Delta_{K-i}(Tx^i+e_i)
+(1-m_K)x^0.
\end{align*}

Let $\epsilon_i=\|e_i\|$.
To track the errors, define
\[
s_{k+1}
=
m_{K-k}(s_k+\epsilon_k)
+\sum_{i=0}^{k-1}\Delta_{K-i}(s_i+\epsilon_i),
\qquad k=0,\ldots,K-1,
\]
where $s_0=0$. Then, by induction, we can verify that
\[
\|x^k-x^\star\|\le \kappa+s_k.
\]
First, $\|x^0-x^\star\|=\kappa+s_0$. Also, we have
\begin{align*}
\|x^{k+1}-x^\star\|
&=
\Biggl\|
m_{K-k}(Tx^k-x^\star+e_k)
+\sum_{i=0}^{k-1}\Delta_{K-i}(Tx^i-x^\star+e_i)
+(1-m_K)(x^0-x^\star)
\Biggr\|\\
&\le
m_{K-k}(\|Tx^k-x^\star\|+\|e_k\|)
+\sum_{i=0}^{k-1}\Delta_{K-i}
  (\|Tx^i-x^\star\|+\|e_i\|)
+(1-m_K)\|x^0-x^\star\|\\
&\le
m_{K-k}(\|x^k-x^\star\|+\|e_k\|)
+\sum_{i=0}^{k-1}\Delta_{K-i}
  (\|x^i-x^\star\|+\|e_i\|)
+(1-m_K)\|x^0-x^\star\|\\
&\le
m_{K-k}(s_k+\epsilon_k)
+\sum_{i=0}^{k-1}\Delta_{K-i}(s_i+\epsilon_i)+\kappa\\
&=s_{k+1}+\kappa,
\end{align*}
where the first inequality follows from the triangle inequality,
the second is from nonexpansiveness and $Tx^\star=x^\star$,
and the third follows from the induction hypothesis.
Then, we also obtain
\[
\|Tx^k-x^0\|\le2\kappa+s_k,
\]
since
$\|Tx^k-x^0\|
\le\|Tx^k-x^\star\|+\|x^\star-x^0\|$, and
\begin{align*}
\|x^{k+1}-x^0\|
&=
\left\|
m_{K-k}(Tx^k-x^0+e_k)
+\sum_{i=0}^{k-1}\Delta_{K-i}(Tx^i-x^0+e_i)
\right\|\\
&\le
m_{K-k}(2\kappa+s_k+\epsilon_k)
+\sum_{i=0}^{k-1}\Delta_{K-i}
  (2\kappa+s_i+\epsilon_i)\\
&=2m_K \kappa +s_{k+1},
\end{align*}
where the inequality uses the previous bound, and the last
equality uses the definition of $s_{k+1}$.

Now define
\[
h_0=2m_K\kappa,
\qquad
h_j=m_{K-j}(h_{j-1}+2\epsilon_{j-1}),
\quad j=1,\ldots,K-1.
\]
Now, we prove the following key lemma by induction on $j$.  
\begin{lemma}
\[
\|x^k-x^j\|\le s_k-s_j+h_j,
\qquad 0\le j<k\le K.
\]
\end{lemma}
\begin{proof}
The case $j=0$ follows from the previous result.

Next, for $1\le j<k\le K$, we have
\begin{align*}
\|x^k-x^j\|
&=
\Biggl\|
m_{K-k+1}\widetilde{T}x^{k-1}
+\sum_{i=0}^{k-2}\Delta_{K-i}\widetilde{T}x^i
+(1-m_K)x^0\\
&\qquad-
\left(
m_{K-j+1}\widetilde{T}x^{j-1}
+\sum_{i=0}^{j-2}\Delta_{K-i}\widetilde{T}x^i
+(1-m_K)x^0
\right)
\Biggr\|\\
&=
\left\|
m_{K-k+1}\widetilde{T}x^{k-1}
+\sum_{i=j}^{k-2}\Delta_{K-i}\widetilde{T}x^i
-m_{K-j}\widetilde{T}x^{j-1}
\right\|\\
&=
\Biggl\|
m_{K-k+1}
  (\widetilde{T}x^{k-1}-\widetilde{T}x^{j-1})
+\sum_{i=j}^{k-2}\Delta_{K-i}
  (\widetilde{T}x^i-\widetilde{T}x^{j-1})
\Biggr\|\\
&\le
m_{K-k+1}
  \|Tx^{k-1}+e_{k-1}-Tx^{j-1}-e_{j-1}\|\\
&\quad+
\sum_{i=j}^{k-2}\Delta_{K-i}
  \|Tx^i+e_i-Tx^{j-1}-e_{j-1}\|\\
&\le
m_{K-k+1}
  (\|x^{k-1}-x^{j-1}\|+\epsilon_{k-1}+\epsilon_{j-1})\\
&\quad+
\sum_{i=j}^{k-2}\Delta_{K-i}
  (\|x^i-x^{j-1}\|+\epsilon_i+\epsilon_{j-1})\\
&\le
m_{K-k+1}
\bigl(
s_{k-1}-s_{j-1}+h_{j-1}
+\epsilon_{k-1}+\epsilon_{j-1}
\bigr)\\
&\quad+
\sum_{i=j}^{k-2}\Delta_{K-i}
\bigl(
s_i-s_{j-1}+h_{j-1}
+\epsilon_i+\epsilon_{j-1}
\bigr)\\
&=
m_{K-k+1}(s_{k-1}+\epsilon_{k-1})
+\sum_{i=j}^{k-2}\Delta_{K-i}(s_i+\epsilon_i)\\
&\quad+
m_{K-j}(-s_{j-1}+h_{j-1}+\epsilon_{j-1}),
\end{align*}
where the inequalities follow from the triangle inequality,
nonexpansiveness, and the induction hypothesis.
% The last equality uses $
% m_{K-k+1}+\sum_{i=j}^{k-2}\Delta_{K-i}=m_{K-j}$.

Similarly, by the recursion for $s_k$,
\[
s_k-s_j
=
m_{K-k+1}(s_{k-1}+\epsilon_{k-1})
+\sum_{i=j}^{k-2}\Delta_{K-i}(s_i+\epsilon_i)
-m_{K-j}(s_{j-1}+\epsilon_{j-1}).
\]
Therefore, we obtain
\begin{align*}
\|x^k-x^j\|
&\le
s_k-s_j+m_{K-j}(h_{j-1}+2\epsilon_{j-1})\\
&=s_k-s_j+h_j.
\end{align*}
\end{proof}
Then, we have
\begin{align*}
\|Tx^K-x^K\|
&=
\Biggl\|
m_1(Tx^K-\widetilde{T}x^{K-1})
+\sum_{i=0}^{K-2}\Delta_{K-i}
  (Tx^K-\widetilde{T}x^i)
+(1-m_K)(Tx^K-x^0)
\Biggr\|\\
&\le
\sum_{i=0}^{K-1}\Delta_{K-i}
  \bigl(\|Tx^K-Tx^i\|+\epsilon_i\bigr)
+(1-m_K)\|Tx^K-x^0\|\\
&\le
\sum_{i=0}^{K-1}\Delta_{K-i}
  \bigl(s_K-s_i+h_i+\epsilon_i\bigr)
+(1-m_K)(2\kappa+s_K)\\
&=
\sum_{i=0}^{K-1}\Delta_{K-i}
  (h_i-s_i+\epsilon_i)
+s_K+2\kappa(1-m_K)\\
&=
\sum_{i=0}^{K-1}\Delta_{K-i}h_i
+2\sum_{i=0}^{K-1}\Delta_{K-i}\epsilon_i
+2\kappa(1-m_K),
\end{align*}
where the last inequality is from the previous lemma and the last equality is from the recursion of $s_K$.
% \[
% s_K=\sum_{i=0}^{K-1}\Delta_{K-i}(s_i+\epsilon_i).
% \]

The recursions for $h_i$ and $W_n$ give
\[
h_i
=
2\kappa\prod_{\ell=K-i}^{K}m_\ell
+
2\sum_{r=0}^{i-1}
\left(\prod_{\ell=K-i}^{K-r-1}m_\ell\right)\epsilon_r,
\qquad
W_n
=
\sum_{j=1}^{n}\Delta_j
\prod_{\ell=j}^{n-1}m_\ell,
\]
where empty sums are zero and empty products are one.
Therefore,
\begin{align*}
\|Tx^K-x^K\|
&\le
2\kappa(1-m_K)
+\sum_{i=0}^{K-1}\Delta_{K-i}h_i
+2\sum_{i=0}^{K-1}\Delta_{K-i}\epsilon_i\\
&=
2\kappa(1-m_K)
+2\kappa\sum_{i=0}^{K-1}\Delta_{K-i}
  \prod_{\ell=K-i}^{K}m_\ell\\
&\quad+
2\sum_{i=0}^{K-1}\sum_{r=0}^{i-1}
\Delta_{K-i}
\left(\prod_{\ell=K-i}^{K-r-1}m_\ell\right)\epsilon_r
+2\sum_{r=0}^{K-1}\Delta_{K-r}\epsilon_r\\
&=
2\kappa \left(
1-m_K
+m_K\sum_{j=1}^{K}\Delta_j
  \prod_{\ell=j}^{K-1}m_\ell
\right)\\
&\quad+
2\sum_{r=0}^{K-1}
\left(
\Delta_{K-r}
+\sum_{i=r+1}^{K-1}\Delta_{K-i}
  \prod_{\ell=K-i}^{K-r-1}m_\ell
\right)\epsilon_r\\
&=
2\kappa(1-m_K+m_KW_K)
+2\sum_{r=0}^{K-1}
\left(
\Delta_{K-r}
+\sum_{j=1}^{K-r-1}\Delta_j
  \prod_{\ell=j}^{K-r-1}m_\ell
\right)\epsilon_r\\
&=
2\kappa(1-m_K+m_KW_K)
+2\sum_{r=0}^{K-1}
\left(
\sum_{j=1}^{K-r}\Delta_j
  \prod_{\ell=j}^{K-r-1}m_\ell
\right)\epsilon_r\\
&=
2\kappa(1-m_K+m_KW_K)
+2\sum_{r=0}^{K-1}W_{K-r}\epsilon_r.
\end{align*}

Lastly, let $m_r=r/(r+2)$.
Then, $W_r \le \frac{2}{r+2}$, since the recursion for $W_r$ and induction give
\begin{align*}
W_r
&=
\frac{2}{(r+1)(r+2)}
+\frac{r-1}{r+1}W_{r-1}\\
&\le
\frac{2}{(r+1)(r+2)}
+\frac{2(r-1)}{(r+1)^2}\\
&\le\frac{2}{r+2}.
\end{align*}
Also, using the bound for $B_K=1-m_K+m_KW_K$ proved in Theorem~\ref{thm:dh}, we get
\[
1-m_K+m_KW_K\le\frac{4}{K+4}.
\]
\end{proof}

\subsection{Comparison with the error bound of Inexact Halpern iteration}

Define the inexact Halpern iteration as
\[
x^{k+1}=m_{k+1}\widetilde{T}x^k+(1-m_{k+1})x^0.
\]
% \begin{fact}
% Assume Assumption~\ref{assum_fixp} holds and that
% $\widetilde{T}x\in\cD$ for all $x\in\cD$.
% Let $K\ge1$, $x^0\in\cD$, and
% $0=m_0\le m_1\le\cdots\le m_K\le1$.
% Suppose $Tx^k+e_k\in\cD$ for $k=0,\ldots,K-1$.
% Then the inexact Halpern iteration satisfies
% \[
% \|Tx^K-x^K\|
% \le
% \kappa(1-m_K+m_KW_K)
% +
% 2\sum_{i=0}^{K-1}
% \left(\prod_{j=i+1}^{K}m_j\right)\|e_i\|,
% \]
% where
% \[
% W_0=0,
% \qquad
% W_i=m_i-m_{i-1}+m_{i-1}W_{i-1},
% \quad i=1,\ldots,K.
% \]
% Specifically, if $m_k=k/(k+2)$, then
% \[
% \|Tx^K-x^K\|
% \le
% \frac{4\kappa}{K+4}
% +
% \frac{2}{(K+1)(K+2)}
% \sum_{i=0}^{K-1}(i+1)(i+2)\|e_i\|.
% \]
% \end{fact}
where
$\widetilde{T}x^k=Tx^k+e_k$. By adapting the analysis in \cite{bravo2026stochastic}
to our inexact setup, we obtain the following result.

\begin{fact}[\cite{bravo2026stochastic}]\label{fact_ih}
Assume Assumption~\ref{assum_fixp}.
Let
$0=m_0\le m_1\le\cdots\le m_K\le1$.
Then the inexact Halpern iteration satisfies
\[
\|Tx^K-x^K\|
\le
2(1-m_K+m_KW_K)\kappa
+
4\sum_{i=0}^{K-1}
\left(\prod_{j=i+1}^{K}m_j\right)\|e_i\|,
\]
where
\[
W_0=0,
\qquad
W_i=m_i-m_{i-1}+m_{i-1}W_{i-1},
\quad i=1,\ldots,K.
\]
Specifically, if $m_k=k/(k+2)$, then
\[
\|Tx^K-x^K\|
\le
\frac{8\kappa}{K+4}
+
\frac{4}{(K+1)(K+2)}
\sum_{i=0}^{K-1}(i+1)(i+2)\|e_i\|.
\]
\end{fact}
When $m_k=k/(k+2)$, as Theorem~\ref{thm:idh} and Fact~\ref{fact_ih} show, the error weights in the bounds
are proportional to
$\frac{(i+1)(i+2)}{(K+1)(K+2)}$ for Halpern and
$\frac{1}{K-i+2}$ for Dual Halpern. This implies
$O(K\epsilon)$ and $O(\log(K+1)\epsilon)$ error terms under a uniform error bound
($\|e_i\|\le\epsilon$), while they exhibit the same residual bound in the exact setup.
% This difference comes from the different error-propagation
% mechanisms of the two iterations, 
As we will see in the next section, this robust error bound of inexact dual Halpern iteration
leads to optimal oracle complexity in the stochastic oracle setup for $p=\infty$,
up to a logarithmic dimension factor.
\subsection{A simple numerical illustration}
Motivated by \cite{bravo2026minimax}, we consider a toy example with
a nonexpansive linear map $Tx=Ax$ on $(\real^2,\infn{\cdot})$, where
\[
A = \frac{1}{|\cos\theta| + |\sin\theta|}
\begin{pmatrix}
\cos\theta & -\sin\theta \\
\sin\theta & \cos\theta
\end{pmatrix},
\]
 and a fixed point is the origin. We set $\theta=\pi/2$, $x^0=(1,0)$, the coefficients for both methods
to $m_k=k/(k+2)$, and the total number of iterations to $K=400$.
% The horizon $K$ for dual Halpern iteration remains fixed throughout each run.

\begin{figure}
    \centering
    \includegraphics[width=1\linewidth]{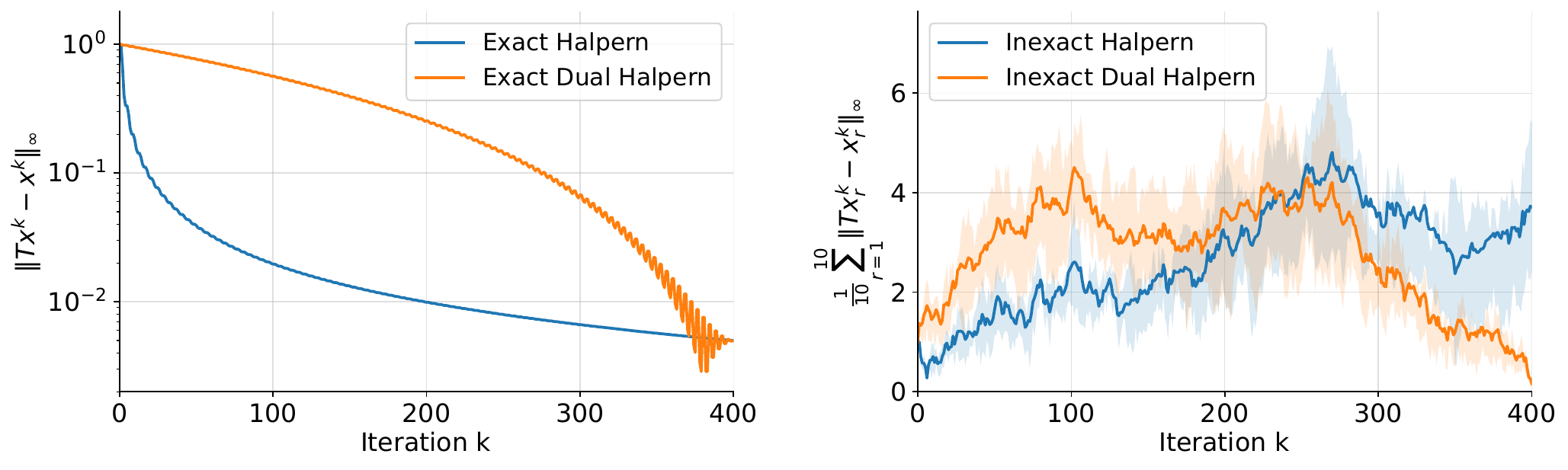}
    \caption{
    Comparison of Halpern iteration and dual Halpern iteration.
    In the exact setting (left), both methods attain the same final
    residual $\infn{Tx^K-x^K}\approx 0.004975$.
    In the inexact setting (right), the final residuals averaged
    over $10$ independent trials,
    $\frac{1}{10}\sum_{r=1}^{10}\infn{Tx_r^K-x_r^K}$,
    are approximately $3.712368$ for inexact Halpern iteration and
    $0.163416$ for inexact dual Halpern iteration.
    The shaded regions indicate the empirical 25th--75th percentiles
    of the residuals across trials.
    }
    \label{fig:placeholder}
\end{figure}

In the deterministic setting, the left panel of
Figure~\ref{fig:placeholder} shows that  Halpern iteration reduces the
residual steadily, whereas dual Halpern iteration reduces it more
sharply near the end of the run, and both methods attain the same
final residual (in particular, for linear operators,
the last iterates of Halpern and dual Halpern iterations are identical 
\cite{yoon2024optimal}).

In the inexact setting, we add uniformly distributed errors to each
operator evaluation. Specifically,
\[
\widetilde{T} x^k=Tx^k+e_k,
\qquad e_k=(a_k,b_k),
\qquad a_k,b_k\sim\operatorname{Unif}[-1/2,1/2],
\]
where the coordinates are sampled independently across iterations
and trials. Within each trial, both methods receive the same error
vector at each iteration. We run $10$ independent trials and plot
the mean true residuals. The right panel of
Figure~\ref{fig:placeholder} shows that during the early iterations, inexact
Halpern iteration exhibits a smaller mean residual than inexact
dual Halpern iteration. Near the final iteration, notably, inexact
dual Halpern iteration achieves a substantially smaller mean residual
and a narrower empirical interquartile band.

\section{Stochastic Dual Halpern iteration}
In this section, we apply the preceding result for the
Inexact Dual Halpern iteration to the following stochastic
oracle setup.

\begin{assumption}\label{assum_stoo}
The stochastic oracle $\widetilde{T}(x,\xi)$ satisfies
% , for some $\sigma\ge0$,
\[
\expec[\widetilde{T}(x,\xi)]=T(x),
\qquad
\expec[\|\widetilde{T}(x,\xi)-T(x)\|^2]\le\sigma^2
\qquad \forall x\in\cX.
\]
% where $\sigma\ge0$.
\end{assumption}

Now, we introduce our Stochastic Dual Halpern iteration.
\begin{algorithm}[H]
\caption{Stochastic Dual Halpern iteration}
\label{alg:spmd}
\begin{algorithmic}[1]
\State \textbf{Input:}
 $x^0 \in\cX$, $N$, $\{k_n\}_{n=0}^{N-1}$, $\{m_n\}_{n=1}^{N}$.
\State Set $\widetilde{T}x^{-1}:=x^0$.
\For{$n=0,1,\ldots,N-1$}
    \State Draw independent samples at $x^n$:
    $\{\widetilde{T}(x^n,\xi_{n,i})\}_{i=1}^{k_n}$.
    \State
    $\displaystyle
    \widetilde{T}x^n
    =\frac1{k_n}\sum_{i=1}^{k_n}
    \widetilde{T}(x^n,\xi_{n,i})
    $.
    \State
    $\displaystyle
    x^{n+1}
    =
    x^n+m_{N-n}
    (\widetilde{T}x^n-\widetilde{T}x^{n-1})$.
\EndFor
\State \textbf{Output:} $x^N$.
\end{algorithmic}
\end{algorithm}
Note that $\widetilde{T}x^n$ corresponds to the inexact operator evaluation which we studied in the previous section.

The following fact bounds the expected norm of an average
of independent zero-mean random vectors in $(\real^d, \|\cdot\|_p)$. In particular, we refer to the summarized result in \cite{diakonikolas2026oracle}.
\begin{fact}[\cite{mascioni1988weak, ledoux1991probability, ball1994sharp, juditsky2008large}]\label{fact_sam}
For independent zero-mean random vectors
$\zeta_1,\ldots,\zeta_m\in\real^d$ satisfying
$\expec[\|\zeta_i\|_p^2]\le\sigma^2$ with $p\in(1,\infty]$,
\[
\expec\left[
\left\|\frac1m\sum_{i=1}^m\zeta_i\right\|_p
\right]
\le C_p\frac{\sigma}{m^{1/q_p}},
\]
where
\[
q_p=
\begin{cases}
\displaystyle \frac{p}{p-1}, & 1<p\le2,\\[6pt]
2, & 2<p\le\infty,
\end{cases}
\qquad
C_p=
\begin{cases}
O(1), & 1<p\le2,\\[4pt]
O(\sqrt{p-1}), & 2<p<\infty,\\[4pt]
O(\sqrt{\log(2d)}), & p=\infty.
\end{cases}
\]
\end{fact}

Based on Fact~\ref{fact_sam}, the following lemma guarantees $\widetilde{T}x\approx Tx$ in expectation if we use sufficiently many samples.
\begin{lemma}\label{lem_stoo}
Assume Assumptions~\ref{assum_fixp} and~\ref{assum_stoo}. Let $(\cX, \|\cdot\|)= (\real^d, \|\cdot\|_p)$ where $p\in(1,\infty]$.
Let $\epsilon>0$.
With
\[
n=
\max\left\{
1,\left\lceil
\left(\frac{C_p\sigma}{\epsilon}\right)^{q_p}
\right\rceil
\right\}
\]
independent samples, we have
\[
\expec\left[
\left\|\frac1n\sum_{i=1}^n\widetilde{T}(x,\xi_i)
-T(x)
\right\|
\right]
\le\epsilon
\qquad \forall x\in\cX.
\]
\end{lemma}

\begin{proof}By Fact~\ref{fact_sam},
\begin{align*}
\expec\left[
\left\|
\frac1n\sum_{i=1}^n
\bigl(\widetilde{T}(x,\xi_i)-T(x)\bigr)
\right\|
\right]
&\le C_p\sigma n^{-1/q_p}\\
&\le\epsilon.
\end{align*}
\end{proof}

Now, we present our main theorem in the stochastic oracle setup which leads to optimal oracle complexity of dual Halpern iteration for $p=\infty$. 
% up to a logarithmic dimension factor.
\begin{theorem}\label{thm:sdh}
Assume Assumptions~\ref{assum_fixp} and~\ref{assum_stoo}. Let $(\cX, \|\cdot\|)= (\real^d, \|\cdot\|_p)$ where $p\in(1,\infty]$.
Let $\epsilon\in(0,\kappa]$ and $q=q_p$. Then, our Stochastic Dual Halpern iteration with
\[
N=\left\lceil\frac{16\kappa}{\epsilon}\right\rceil,
\qquad
m_j=\frac{j}{j+2},
\quad j=1,\ldots,N,
\]
and
\[
k_n=
\max\left\{
1,\left\lceil
\left(\frac{8(q+1)C_p\sigma}{\epsilon}\right)^q
\left(\frac{N}{N-n+2}\right)^{q/(q+1)}
\right\rceil
\right\},
\quad n=0,\ldots,N-1,
\]
satisfies
\[
\expec[\|Tx^N-x^N\|]\le\epsilon,
\]
with a total number of samples 
% (where the constants implicit in $O$ may depend on $p$ for $1<p<2$)
\[
\sum_{n=0}^{N-1}k_n=
\begin{cases}
\displaystyle
O\left(
\frac{\kappa}{\epsilon}
+
\frac{\kappa\sigma^{p/(p-1)}}
{\epsilon^{(2p-1)/(p-1)}}
\right),
&1<p<2,\\[10pt]
\displaystyle
O\left(
\frac{\kappa}{\epsilon}
+
\frac{\kappa(p-1)\sigma^2}{\epsilon^3}
\right),
&2\le p<\infty,\\[10pt]
\displaystyle
O\left(
\frac{\kappa}{\epsilon}
+
\frac{\kappa\sigma^2\log(2d)}{\epsilon^3}
\right),
&p=\infty.
\end{cases}
\]
\end{theorem}

\begin{proof}
Write $q=q_p$ and let $
e_n:=\|\widetilde{T}x^n-Tx^n\|$. Also let $\mathcal F_n$ denote the history before drawing
the samples at iteration $n$.
Conditional on $\mathcal F_n$, the point $x^n$ is fixed
and the new samples are independent.
Thus,
\[
\expec[e_n\mid\mathcal F_n]
\le C_p\sigma k_n^{-1/q},
\qquad
\expec[e_n]\le C_p\sigma k_n^{-1/q},
\]
by Lemma~\ref{lem_stoo}. By Theorem~\ref{thm:idh},
\[
\|Tx^N-x^N\|
\le
\frac{8\kappa}{N+1}
+
4\sum_{n=0}^{N-1}\frac{e_n}{N-n+2},
\]
and taking expectations gives
\[
\expec[\|Tx^N-x^N\|]
\le
\frac{8\kappa}{N+1}
+
4C_p\sigma
\sum_{n=0}^{N-1}
\frac{k_n^{-1/q}}{N-n+2}.
\]
The choice of $N$ gives
\[
\frac{8\kappa}{N+1}\le\frac{\epsilon}{2}.
\]

The choice of $k_n$ gives
\[
C_p\sigma k_n^{-1/q}
\le
\frac{\epsilon}{8(q+1)}
\left(\frac{N-n+2}{N}\right)^{1/(q+1)},
\]
and therefore,
\begin{align*}
4C_p\sigma
\sum_{n=0}^{N-1}
\frac{k_n^{-1/q}}{N-n+2}
&\le
\frac{\epsilon}{2(q+1)N^{1/(q+1)}}
\sum_{u=1}^{N}(u+2)^{-q/(q+1)}\\
&\le\frac{\epsilon}{2},
\end{align*}
where the second inequality follows from
\[
\sum_{u=1}^{N}(u+2)^{-q/(q+1)}
\le
\int_0^N t^{-q/(q+1)}\,dt
=
(q+1)N^{1/(q+1)}.
\]
This proves that $\expec[\|Tx^N-x^N\|]\le\epsilon$.

Finally, using
$\max\{1,\lceil a\rceil\}\le1+a$ for $a\ge0$,
\begin{align*}
\sum_{n=0}^{N-1}k_n
&\le
N+
\left(\frac{8(q+1)C_p\sigma}{\epsilon}\right)^q
N^{q/(q+1)}
\sum_{u=1}^{N}(u+2)^{-q/(q+1)}\\
&\le
N\left[
1+(q+1)
\left(\frac{8(q+1)C_p\sigma}{\epsilon}\right)^q
\right].
\end{align*}
Since $N=O(\kappa/\epsilon)$ for $\epsilon\le\kappa$, we conclude the proof.
% substituting the values of $q_p$ and the bounds on $C_p$
% gives the stated sample bound.
\end{proof}
% We note that, for $1<p<2$, the constant implicit in $O$ may depend on $p$.

\subsection{Comparison with the oracle complexity of stochastic Halpern iteration}

In the literature on stochastic oracle complexity in normed spaces,
our Theorem~\ref{thm:sdh} improves
the dependence on $\epsilon$ in the oracle complexity
bounds of the recursively regularized Halpern iteration
\cite{diakonikolas2026oracle} by logarithmic factors in $1/\epsilon$.
For $p=\infty$, our result
$O\left(
\frac{\kappa}{\epsilon}
+
\frac{\kappa\sigma^2\log(2d)}{\epsilon^3}
\right)$
matches the lower bound\footnote{
% The lower bound in \cite{diakonikolas2026oracle} is established
% for nonexpansive self-maps of compact convex boxes. 
% % Rescaling their construction to $\cX=[-\kappa,\kappa]^d$ and taking $x^0=0$ gives a lower bound for the class of problems satisfying $\|x^0-x^\star\|_\infty\le\kappa$.
% To keep our iterates in a compact convex set $\cX$, we can apply
% an $L$-Lipschitz retraction $R:\real^d\to\cX$, satisfying
% $R(x)=x$ for all $x\in\cX$, to each minibatch operator estimate
% $\overline T(x)$ as in \cite{diakonikolas2026oracle}. 
% % Since $Tx\in\cX$, $\|R(\overline T(x))-Tx\|\le L\|\overline T(x)-Tx\|$.
% % Together with the convex-combination form of dual Halpern iteration and $x^0\in\cX$, 
% This keeps the iterates in $\cX$ and preserves
% the $O(1/\epsilon^3)$ dependence for fixed $L$ through our inexact
% error analysis. 
% Their constant-probability lower bound also implies an expectation lower bound of the same order after rescaling $\epsilon$ by a constant factor.
The lower bound in \cite{diakonikolas2026oracle} is established
for nonexpansive self-maps of compact convex boxes.
To keep our iterates in a compact convex set $\cX$, we can use
an $L$-Lipschitz retraction $R:\real^d\to\cX$, satisfying
$R(x)=x$ for all $x\in\cX$,
and apply it to each minibatch operator estimate $\widetilde{T} (x)$ as in \cite{diakonikolas2026oracle}.
Provided that $x^0\in\cX$ and $T(\cX)\subseteq\cX$, the
convex-combination form of dual Halpern iteration ensures that
the iterates remain in $\cX$, and our inexact error analysis
preserves the $O(1/\epsilon^3)$ dependence for fixed $L$. 
% Their constant-probability lower bound also implies an expectation lower bound of the same order after rescaling $\epsilon$ by a constant factor.
}
$\Omega\left(
\frac{\kappa}{\epsilon}
+
\frac{\kappa\sigma^2}{\epsilon^3}
\right)$
in sufficiently high dimensions, up to the logarithmic dimension
factor, and is therefore optimal in its dependence on $\epsilon$.
Compared with stochastic Halpern iteration without recursive 
variance reduction in \cite{bravo2026stochastic}, our result improves oracle complexity from $\widetilde{O}(\epsilon^{-5})$
to $O(\epsilon^{-3})$ for $p\in[2,\infty]$.
Table~\ref{table} summarizes the stochastic oracle complexities
of prior works in $(\real^d,\infn{\cdot})$.

% \subsection{Comparison with oracle complexity of stochastic Halpern iteration}
% In the literature on stochastic oracle complexity in normed spaces, our sample complexity improves upon the best knonw stochastic oracle complexity bounds of recursively regularzied Halpern iteration \cite{diakonikolas2026oracle} by logarithmic factors in $1/\epsilon$. For $p=\infty$, our result $O\left(
% \frac{\kappa}{\epsilon}
% +
% \frac{\kappa\sigma^2\log(2d)}{\epsilon^3}
% \right)$ matches the lower bound\footnote{The lower bound in \cite{diakonikolas2026oracle} is estbalished with respecto  selfmap $T:\cX \rightarrow \cX$ where $\cX$ where $\cX$ is compact. We note that using $L$-lipsthiz retraction operator $R:\real^d \rightarrow \cX$ satisfying $R(x)=x, \forall x\in\cX$, we can force iterate staying in $\cX$ as \cite{diakonikolas2026oracle} did, and we can easilty verifeid $O(1/\epsilon^3)$ optimal stocahsic complextiy in this setup so match the lower bound.} $\Omega\left(
% \frac{\kappa}{\epsilon}
% +
% \frac{\kappa\sigma^2}{\epsilon^3}
% \right)$ in sufficiently high dimension, up to the logarithmic dimension factor, and so is optimal in its dependence on $\epsilon$.  Compared to stochastic Halpern iteration without variance reduction in \cite{bravo2026stochastic}, our result improves the stochastic oracle complexity by a factor of $\epsilon^{-2}$ for $p\in[2,\infty]$. Table~\ref{table} summarizes storchastic oracle complextiy of prior works in $(\real^d, \infn{\cdot})$.
% We believe this reveals the robust error bound of dual Halpern iteration.  

Lastly, we note that in the literature on stochastic oracle complexity of dual Halpern iteration, our result generalizes prior oracle complexity results of dual Halpern iteration from Hilbert spaces with a mean-square nonexpansiveness condition \cite{yoon2026direct} to these finite-dimensional normed spaces without additional assumption.  

% Requires \usepackage{booktabs,tabularx}.
% Replace citation keys with those in your bibliography.

\begin{table*}[h!]
\centering
\small
\renewcommand{\arraystretch}{1.2}
\caption{
Oracle complexity for an $\epsilon$-residual
for a $\|\cdot\|_\infty$-nonexpansive operator under Assumption~\ref{assum_stoo}.
Only the dependence on $\epsilon$ is displayed, and the notation $\widetilde{O}$ hides additional logarithmic factors. Expectation guarantees concern
$\mathbb E[\|Tx-x\|_\infty]\le\epsilon$;
high-probability guarantees concern
$\mathbb P(\|Tx-x\|_\infty\le\epsilon)\ge1-\delta$.
% The lower bound applies in sufficiently high dimension.
}
\label{table}
\begin{tabularx}{\textwidth}{
@{}X c c c@{}
}
\toprule
Reference
& oracle complexity
& Guarantee
& \shortstack{Type of Algorithm}
\\
\midrule

\cite{bravo2026stochastic},
Corollary~3.5
& $\widetilde{O}(\epsilon^{-5})$
& Expectation
& \shortstack{Stochastic Halpern Iteration}
\\
\addlinespace

\cite{diakonikolas2026solving},
Corollary~1
& $\widetilde{O}(\epsilon^{-5})$
& High probability
&Variance-reduced Halpern Iteration
\\
\addlinespace

\cite{diakonikolas2026oracle},
Theorem~5
& $\widetilde{O}(\epsilon^{-3})$
& High probability
&Recursively regularized Halpern iteration
\\
\addlinespace

This work, Theorem~\ref{thm:sdh}
& $O(\epsilon^{-3})$
& Expectation
&Stochastic Dual Halpern iteration
\\

\midrule
\cite{diakonikolas2026oracle},
Theorem~13
& $\Omega(\epsilon^{-3})$
& Constant probability
&Lower bound
\\

\bottomrule
\end{tabularx}

\medskip
\begin{minipage}{\textwidth}
\end{minipage}
\end{table*}

\section{Conclusion}

In this work, we studied the convergence of Dual Halpern Iteration with nonexpansive operators in exact (deterministic), inexact, and stochastic setups. We showed that Dual Halpern Iteration shows a robust error bound, particularly compared to Halpern iteration, and achieves optimal oracle complexity in the deterministic setup in normed spaces and in the stochastic oracle setup for $p=\infty$. 
% up to the logarithmic dimension factor in the latter case.

Our work opens the door to several directions for future work. One future research direction is to reduce the logarithmic dimension factor in Theorem~\ref{thm:sdh} or improve the lower bound to match the dimension dependence. Another direction is to extend dual Halpern iteration to a prior-knowledge-free method which does not require knowledge of problem parameters to obtain an $\epsilon$-residual, as in \cite{diakonikolas2026solving, lee2025near}. Extending dual Halpern iteration to a more general Markovian oracle setup \cite{kaya2026banach,lee2026learning} would also be an interesting direction.

Lastly, we highlight that dual Halpern iteration could be broadly applied to various stochastic setups, including reinforcement learning, min-max optimization, and machine learning, as Halpern Iteration has been \cite{lee2021fast, lee2025optimal, cai2024variance, colao2026convergence}. We believe that the robust error bound of dual Halpern iteration provides promising performance for such applications.

% \paragraph{AI Disclosure Statement}

% In this work, the authors used AI tools for language editing, proofreading, and refining calculations. We have not used
% generative AI tools for the development of the core methodology and theoretical results.

\newpage

\bibliographystyle{plain}
% \bibliography{references}
\bibliography{dual_halpern}

@article{lee2026m,
  title={M-duality in fixed-point problems for general normed spaces},
  author={Lee, Jongmin and Woo, Sangyeop and Ryu, Ernest K},
  journal={arXiv preprint arXiv:2609.08151},
  year={2026}
}

@article{bravo2026minimax,
  title={Minimax-optimal Halpern iterations for Lipschitz maps},
  author={Bravo, Mario and Cominetti, Roberto and Lee, Jongmin},
  journal={Mathematical Programming},
  pages={1--32},
  year={2026},
  publisher={Springer}
}

@article{lee2025optimal,
  title={Optimal non-asymptotic rates of value iteration for average-reward markov decision processes},
  author={Lee, Jongmin and Ryu, Ernest},
  journal={International Conference on Learning Representations},
  year={2025}
}

@article{lee2026learning,
  title={Learning Policy from a Single Trajectory in Average-Reward Markov Decision Process},
  author={Lee, Jongmin and Ryu, Ernest K and Aggarwal, Vaneet},
  journal={arXiv preprint arXiv:2606.16729},
  year={2026}
}

@article{sabach2017first,
  title={A first order method for solving convex bilevel optimization problems},
  author={Sabach, Shoham and Shtern, Shimrit},
  journal={SIAM Journal on Optimization},
  volume={27},
  number={2},
  pages={640--660},
  year={2017}
}

@article{cai2022stochastic,
  title={Stochastic Halpern iteration with variance reduction for stochastic monotone inclusions},
  author={Cai, Xufeng and Song, Chaobing and Guzm{\'a}n, Crist{\'o}bal and Diakonikolas, Jelena},
  journal={Neural Information Processing Systems},
  volume={35},
  pages={24766--24779},
  year={2022}
}

@article{lee2021fast,
  title={Fast extra gradient methods for smooth structured nonconvex-nonconcave minimax problems},
  author={Lee, Sucheol and Kim, Donghwan},
  journal={Neural Information Processing Systems},
  volume={34},
  pages={22588--22600},
  year={2021}
}

@article{yoon2026dual,
  title={Dual-Anchor Acceleration Is Near-Optimal for Stochastic Monotone Root-Finding},
  author={Yoon, TaeHo and Loizou, Nicolas},
  journal={arXiv preprint arXiv:2609.36033},
  year={2026}
}

@article{yoon2026direct,
  title={Direct Acceleration of Stochastic Root-Finding Without Variance Reduction and Regularization},
  author={Yoon, TaeHo and Loizou, Nicolas},
  journal={arXiv preprint arXiv:2608.12043},
  year={2026}
}

@article{yoon2026toward,
  title={Toward a Systematic Understanding and Interactive Search of {L}yapunov-Style Proofs in Optimization},
  author={Yoon, TaeHo and Suh, Jaewook J and Nguyen, Edward Duc Hien and Ying, Bicheng and Ma, Shiqian},
  journal={arXiv preprint arXiv:2606.26077},
  year={2026}
}

@article{diakonikolas2022potential,
  title={Potential function-based framework for minimizing gradients in convex and min-max optimization},
  author={Diakonikolas, Jelena and Wang, Puqian},
  journal={SIAM Journal on Optimization},
  volume={32},
  number={3},
  pages={1668--1697},
  year={2022},
  publisher={SIAM}
}

@article{cai2024variance,
  title={Variance reduced halpern iteration for finite-sum monotone inclusions},
  author={Cai, Xufeng and Alacaoglu, Ahmet and Diakonikolas, Jelena},
  journal={International Conference on Learning Representations},
  year={2024}
}

@article{chen2024near,
  title={Near-optimal algorithms for making the gradient small in stochastic minimax optimization},
  author={Chen, Lesi and Luo, Luo},
  journal={Journal of Machine Learning Research},
  volume={25},
  number={387},
  pages={1--44},
  year={2024}
}

@article{
    Lieder2021halpern,
    title={On the convergence rate of the {H}alpern-iteration},
    author={Felix Lieder},
    journal={Optimization Letters},
    volume={15},
    number={2},
    pages={405--418},
    year={2021}
}

@article{halpern1967fixed,
  title={Fixed points of nonexpanding maps},
  author={Halpern, Benjamin},
  journal={Bulletin of the American Mathematical Society},
  volume={73},
  number={6},
  pages={957--961},
  year={1967},
  publisher={American Mathematical Society}
}

@article{park2022exact,
  title={Exact Optimal Accelerated Complexity for Fixed-Point Iterations},
  author={Park, Jisun and Ryu, Ernest K.},
  journal={International Conference on Machine Learning},
  year={2022}
}

@book{ledoux1991probability,
  title={Probability in Banach Spaces: isoperimetry and processes},
  author={Ledoux, Michel and Talagrand, Michel},
  volume={23},
  year={1991},
  publisher={Springer}
}

@article{juditsky2008large,
  title={Large deviations of vector-valued martingales in 2-smooth normed spaces},
  author={Juditsky, Anatoli and Nemirovski, Arkadii S},
  journal={arXiv preprint arXiv:0809.0813},
  year={2008}
}

@article{ball1994sharp,
  title={Sharp uniform convexity and smoothness inequalities for trace norms},
  author={Ball, Keith and Carlen, Eric A and Lieb, Elliott H},
  journal={Inventiones mathematicae},
  volume={115},
  number={1},
  pages={463--482},
  year={1994},
  publisher={Springer}
}

@article{mascioni1988weak,
  author  = {Mascioni, Vania},
  title   = {On Weak Cotype and Weak Type in Banach Spaces},
  journal = {Note di Matematica},
  volume  = {8},
  number  = {1},
  year    = {1988},
  pages   = {67--110}
}

@article{diakonikolas2026solving,
  title={Solving Stochastic Fixed-Point Equations with High Probability},
  author={Diakonikolas, Jelena},
  journal={arXiv preprint arXiv:2607.09097},
  year={2026}
}

@article{colao2026convergence,
  title={On the Convergence of HalpernSGD},
  author={Colao, Vittorio and Foglia, Katherine Rossella},
  journal={arXiv preprint arXiv:2601.18906},
  year={2026}
}

@article{kaya2026banach,
  title={A Banach-Space Theory of Markovian Halpern Iteration for Non-Expansive Maps},
  author={Kaya, Ege C and Fazla, Arda and Sahin, M Berk and Hashemi, Abolfazl},
  journal={arXiv preprint arXiv:2608.15966},
  year={2026}
}

@article{lee2025near,
  title={Near-Optimal Sample Complexity for {MDP}s via Anchoring},
  author={Lee, Jongmin and Bravo, Mario and Cominetti, Roberto},
  journal={International Conference on Machine Learning},
  year={2025}
}

@article{shu2026unified,
  title={A Unified Theory of H-Duality in First-Order Methods},
  author={Shu, Kevin and Wang, Alex L},
  journal={arXiv preprint arXiv:2609.03281},
  year={2026}
}

@article{yoon2024optimal,
  title={Optimal acceleration for minimax and fixed-point problems is not unique},
  author={Yoon, Taeho and Kim, Jaeyeon and Suh, Jaewook J and Ryu, Ernest K},
  journal={Internation conference on machine Learning},
  year={2024}
}

@article{contreras2022optimal,
  title={Optimal error bounds for non-expansive fixed-point iterations in normed spaces},
  author={Contreras, Juan Pablo and Cominetti, Roberto},
  journal={Mathematical Programming},
  volume= {199},
number={1--2},
pages= {343--374},
  year={2022}
}

@article{diakonikolas2026oracle,
  title={Oracle Complexity of Stochastic Fixed-Point Equations with Nonexpansive Maps},
  author={Diakonikolas, Jelena and Guzm{\'a}n, Crist{\'o}bal and Mart{\'\i}nez-Rubio, David},
  journal={arXiv preprint arXiv:2609.09524},
  year={2026}
}

@article{bravo2026stochastic,
  title={Stochastic Halpern iteration in normed spaces and applications to reinforcement learning: M. Bravo, JP Contreras},
  author={Bravo, Mario and Contreras, Juan Pablo},
  journal={Mathematical Programming},
  pages={1--39},
  year={2026},
  publisher={Springer}
}

@article{lee2024accelerating,
  title={Accelerating value iteration with anchoring},
  author={Lee, Jongmin and Ryu, Ernest},
  journal={Neural Information Processing Systems},
  year={2023}
}

\end{document}